\documentclass[11pt]{amsart}

\usepackage{geometry,amssymb,latexsym,xcolor,enumerate,graphicx}
\usepackage[T2A,T1]{fontenc}
\usepackage[utf8]{inputenc}
\usepackage[russian,english]{babel}

\title[Boundary vorticity of
incompressible 2D flows]{
Boundary vorticity of
incompressible 2D flows}
\author{Giovanni Franzina}
\numberwithin{equation}{section}
\newtheorem{theorem}{Theorem}
\newtheorem{proposition}{Proposition}
\newtheorem{lemma}{Lemma}

\theoremstyle{definition}
\newtheorem{rmk}{Remark}

\usepackage[
colorlinks=true,urlcolor=cyan,
citecolor=magenta,linkcolor=cyan,
linktocpage,pdfpagelabels,bookmarksnumbered,bookmarksopen,
pagebackref
]{hyperref}

\address[G. Franzina]{\quad\newline
\indent Istituto per le Applicazioni del Calcolo ``M. Picone''
\newline\indent
Consiglio Nazionale delle Ricerche
\newline\indent 
Via dei Taurini, 19, 00185 Roma, Italy}


\begin{document}
\maketitle

\begin{abstract}
For a homogeneous incompressible 2D fluid
confined within a
bounded Lipschitz simply connected domain, homogeneous
Neumann pressure boundary conditions are equivalent to a constant boundary vorticity. We investigate the rigidity of such conditions.

\end{abstract}

\section{Introduction}
We consider the two-dimensional Navier-Stokes equations
for the velocity $v=(v^1,v^2)$ of
the plane-parallel flow of a
homogeneous incompressible fluid
within static rigid walls, viz.
\begin{equation}
\label{4}
\begin{cases}
\partial_t v + (v\cdot\nabla)v = 
\nu\Delta v+f-\nabla p
 \,,&\qquad\text{in $\Omega\times(0,+\infty)$,}\\
\nabla\cdot v=0\,,&\qquad\text{in $\Omega\times(0,+\infty)$,}\\
v=0\,,&\qquad\text{on $\partial\Omega\times(0,+\infty)$.}
\end{cases}
\end{equation}
An external force with divergence-free density  $f$ is applied, which vanishes at the boundary.
The kinematic viscosity (divided by Reynolds number)
$\nu$ is positive, hence the no-slip condition, dictated by the last equation,
makes sense at the boundary of $\Omega$. The latter, throughout this paper,  is assumed to be a Lipschitz planar open set of finite area.

Against this backdrop, the pressure $p$ arises in
enforcing the incompressibility constraint, expressed as in
the second equation of \eqref{4}. In view of that, when solving the first one for $p$, we arrive at Poisson equation
$-\Delta p = \nabla\cdot(u\cdot\nabla u)$. 
If the boundary is smooth and prevents any penetration, 
that comes automatically with the condition
$\frac{\partial p}{\partial n} =
\kappa (v\cdot\tau)^2 - \nu \partial_\tau (\nabla\times v)$
for the pressure gradient in the outward normal direction $n$. Here,
 $\kappa $ is the boundary
 curvature 
 and $\partial_\tau$ is the derivative
along the direction $\tau$ parallel to the contour, followed counterclockwise~\cite{Te}.
For positive viscosities,
the tangential derivative of the vorticity, in general,
does contribute to the right hand side, whereas
 the term involving curvature always vanishes for velocities subject to the no-slip condition.

In some special geometries, 
the Neumann boundary data for the pressure gradient
can be calculated explicitly, though. For instance, alongside to the boundaries
of an infinite horizontal straight channel the vertical component of the pressure gradient must vanish. (An additional condition at infinity must supplement the problem in order to make
it well posed, in this case.) Also, it turns out that
a homogeneous Neumann-type condition must hold at the boundaries of a circular domain.

With these two examples in mind, we are led to raise the question: {\em how rigid is the condition that the component of $\nabla p$ normal to the boundary of $\Omega$ vanishes identically?} What seems
natural is the guess that {\em the disc is the only bounded
simply connected domain on which the pressure satisfies homogeneous Neumann boundary conditions.} 
As said, for the incompressible 2D Navier-Stokes system \eqref{4}, in a smooth bounded domain, an equivalent question is how strict is the requirement that the vorticity is constant on the walls.

The present manuscript does not give an answer in the full generality 
of equations \eqref{4}. In order to address this issue with some partial results, we start by investigating the matter
in the simplified set-up of their formal low Reynolds number
limit, i.e., the Stokes problem
\vskip-.4cm
\begin{subequations}\label{Re-zero-over}
\begin{equation}
\label{Re-zero}
\begin{cases}
\nu\Delta u+f=\nabla p \,,&\qquad\text{in $\Omega$,}\\
\nabla\cdot u=0\,,&\qquad\text{in $\Omega$,}\\
u=0\,,&\qquad\text{on $\partial\Omega$.}
\end{cases}
\end{equation}
An overdermination occurs, in this context, by imposing the additional condition that
\begin{equation}\label{Re-zero-plus}
	\nabla p =0 \,,\qquad \text{in $\Omega$.}
\end{equation}
\end{subequations}
Recall that the pressure
in \eqref{4} scales linearly with the inverse Reynolds number, and the function $p$ in \eqref{Re-zero} is harmonic in $\Omega$. Hence, for \eqref{Re-zero}, Neumann pressure
boundary conditions are equivalent to \eqref{Re-zero-plus}.
Then, we consider  the constrained variational system
\vskip-.3cm
\begin{subequations}\label{vortex}
\begin{equation}\label{vortex-1}
\begin{cases}
\nu\Delta w + \nabla\times f =0 \,,\qquad\text{in
$\Omega$,}\\
\displaystyle
\int_\Omega w \,h\,dx =0\,,\qquad
\text{for all $h\in H^1(\Omega)$, with $\Delta h=0$.}
\end{cases}
\end{equation}
We will sometimes endow \eqref{vortex-1} with the additional condition for the boundary vorticity
\begin{equation}\label{vortex-2}
\nu\,w= {\rm const},\qquad \text{on $\partial\Omega$.}
\end{equation}
\end{subequations}
Given $\varphi\in H_0^1(\Omega)$, we  also consider 
the fourth-order problem
\begin{subequations}
\label{corrente}
\begin{equation}
\label{corrente-1}
\begin{cases}
\nu\Delta^2 \psi+\Delta\varphi=0 \,,&\qquad\text{in $\Omega$,}\\
\psi=0\,,&\qquad\text{on $\partial\Omega$,}\\
\frac{\partial\psi}{\partial n}=0\,,&\qquad\text{on $\partial\Omega$.}
\end{cases}
\end{equation}
We shall sometimes focus on solutions of \eqref{corrente-1} subject to the additional condition
\begin{equation}
\label{corrente-2}
\nu\,\Delta \psi = {\rm const},\qquad\text{on $\partial\Omega$.}
\end{equation}
\end{subequations}
 
 \begin{subequations}\label{subASS}
As long as it is assumed to be constant, 
 the viscosity $\nu$ may be removed both from 
 \eqref{vortex-2} and from \eqref{corrente-2}.
We begin by commenting about the
relation between these boundary value problems.
For the vector field $f=(f^1,f^2)\in L^2(\Omega\mathbin{;}\mathbb R^2)$ in \eqref{Re-zero}, we assume that
\begin{equation}
\label{fASS}
\nabla\cdot f=0\,,\qquad
 \partial_{x_1}f^2-\partial_{x_2}f^1 \in H^{-1}(\Omega)\,,\qquad
 \text{$f\cdot n=0$ on $\partial\Omega$}.
\end{equation}
Thanks to the first equation in \eqref{fASS},
the normal trace of $f$ at the boundary is well-defined in a weak sense and the last condition in \eqref{fASS} makes sense. As for the scalar function $\varphi$ in \eqref{corrente-1}, owing to the second requirement in \eqref{fASS}
it is possible to find it such that
\begin{equation}
\label{varphiASS}
\varphi\in H^1_0(\Omega)\quad \text{with }\quad \Delta\varphi+\nabla\times f=0\,.
\end{equation}
\end{subequations}

\begin{proposition}\label{prop0.1}
Let $f$ and $\varphi$ be as in \eqref{subASS}.
If
 $w\in H^1(\Omega)$ is a solution of
 \eqref{vortex-1} and we define $\psi\in H^1_0(\Omega)$
 by requiring that $\Delta \psi =w$, then $\psi\in H^2_0(\Omega)$ and $\psi$
 is a solution of \eqref{corrente-1}. Conversely,
 if $\psi\in H^2_0(\Omega)$ is a solution of
 \eqref{corrente-1} and we set $w= \Delta \psi$, then
 $w\in H^1(\Omega)$ and $w$ is a solution
 of \eqref{vortex-1}. If in addition $\Omega $ is simply
 connected, then,
given a solution  $\psi\in H^2_0(\Omega)$  of \eqref{corrente-1}, the function
 \begin{equation}
 \label{h}
 h=\nu \Delta\psi+\varphi
 \end{equation}
 is harmonic; moreover,
 \begin{equation}
 \label{perp}
 u=-\nabla^\perp \psi = (\partial_{x_2}\psi\mathbin,-\partial_{x_1}\psi)
 \end{equation}
 defines an element of $H^1_0(\Omega;\mathbb R^2)$ which solves \eqref{Re-zero}. The converse also holds: given
 a solution $u\in H^1_0(\Omega;\mathbb R^2)$ of
 \eqref{Re-zero} we can find $\psi\in H^2_0(\Omega)$,
 with \eqref{perp}, that solves \eqref{corrente-1} and is such that \eqref{perp} is harmonic; in this case 
 \begin{equation}
 \label{CR}
 	\nabla p = -\nabla^\perp h\,.
 \end{equation}
\end{proposition}

The part of Proposition~\ref{prop0.1}
that concerns the equivalence of physical 
and of secondary variables is known, 
see~\cite{QV},  and
 we have recalled it here for our convenience in this form. 
In view of the harmonic character of the function \eqref{h}
and of Cauchy-Riemann equations \eqref{CR},
by the maximum principle we draw the following consequence.

\begin{theorem}\label{prop0.2}
Let $f$ and $\varphi$ be as in \eqref{subASS}.
The two overdetermined problems
\eqref{corrente} and \eqref{vortex} are equivalent and,
if in addition $\Omega$ is simply connected, either
is equivalent to the overdetermined
problem \eqref{Re-zero-over}.
\end{theorem}
Here, the holomorphic character of the complex-valued
function $h+\mathsf{i}\,p$ is an expedient ingredient in proving
Theorem~\ref{prop0.2}, clearly implying a technical restriction on the dimension. 
As mentioned incidentally above,
in the special case of a planar {\em smooth} simply connected open set $\Omega$ 
in view of global elliptic estimates~\cite{L,S2}
the equivalence of
\eqref{Re-zero-plus}, \eqref{vortex-2},
and \eqref{corrente-2}
can be also inferred from the explicit form of the normal boundary pressure gradient.

The equivalence of \eqref{Re-zero-over}, \eqref{vortex},
and \eqref{corrente}
motivates the comparison between the eigenvalues of the buckling problem
\begin{equation}
\label{1}
	\begin{cases}
	\Delta^2\psi +\lambda \Delta\psi=0\,,&\qquad\text{in $\Omega$,}\\
	\psi = 0\,,&\qquad \text{on $\partial\Omega$,}\\
	\frac{\partial\psi}{\partial n} =0\,,&\qquad \text{on $\partial\Omega$,}
	\end{cases}
\end{equation}
the eigenvalues relevant to the linear and stationary approximation of hydrodynamics
\begin{equation}
\label{2}
\begin{cases}
\Delta u +\lambda u = \nabla p\,, &\qquad\text{in $\Omega$,}\\
\nabla\cdot u=0 \,, & \qquad\text{in $\Omega$,}\\
u=0\,, &\qquad\text{on $\partial\Omega$,}
\end{cases}
\end{equation}
and the eigenvalues of the following constrained spectral variational problem:
\begin{equation}\label{3}
\begin{cases}
\Delta w + \lambda w=0\,,&\qquad\text{in $\Omega$,}\\
\displaystyle\int_\Omega wh \,dx=0 \,,&\qquad\text{for all $h\in H^1(\Omega)$
with $\Delta h=0$ in $\Omega$.}
\end{cases}
\end{equation}

After denoting by $\mathfrak{S}^B(\Omega),\mathfrak{S}^S(\Omega),
\mathfrak{S}^H(\Omega)$
the sets of all $\lambda>0$ for which \eqref{1}, \eqref{2}, and \eqref{3} admit a non-trivial solution, respectively, by Proposition~\ref{prop0.1} we see that

\[
\mathfrak{S}^B(\Omega)=
\mathfrak{S}^H(\Omega)
\subseteq
\mathfrak{S}^S(\Omega)\,,
\]
with the second inclusion being an equality
if $\Omega$ is simply connected.
Then, we consider the special value
represented by the minimum of
the buckling spectrum $\mathfrak{S}^B(\Omega)$, i.e., 
\begin{equation}
\label{lambda1B}
\lambda_1^B(\Omega) = \min_{\psi\in H_0^2(\Omega)}\frac{
	\displaystyle\int_\Omega(\Delta\psi)^2\,dx}{
	\displaystyle\int_\Omega|\nabla\psi|^2\,dx}\,,
\end{equation}
and its distinguished r\^ole in Schiffer's overdetermined eigenvalue problem
\begin{equation}
\label{5}
\begin{cases}
\Delta w + \lambda w=0\,,&\qquad\text{in $\Omega$,}\\
w\equiv{\rm const}\,,&\qquad\text{on $\partial\Omega$,}\\
\frac{\partial w}{\partial n} =0 \,,&\qquad\text{on $\partial\Omega$.}
\end{cases}
\end{equation}
Let $\mathfrak{S}^D(\Omega)=\{\lambda_i^D(\Omega)\}_{i\ge1}$
be the set of all eigenvalues of the Dirichlet Laplacian on $\Omega$.
%We recall the following fact.

\begin{theorem}\label{thm1}
Let $\Omega$ be a smooth and bounded planar open set. Then,  \eqref{5} has non-zero solutions if and only if
$\lambda\in\mathfrak{S}^B(\Omega)$
and \eqref{1} has a non-zero solution $\psi$ with
$\Delta\psi$ constant along $\partial\Omega$.
In that case, either $\lambda_1^B(\Omega)<\lambda=\lambda_i^D(\Omega)$, with $i>2$, or $\Omega$ must be a disc.
\end{theorem}

The smoothness assumption on $\Omega$ in Theorem~\ref{thm1} is redundant: it is known~\cite{W}
that the existence of 
non-zero solutions for the overdetermined problem \eqref{5} implies the real analytic regularity of the boundary.

The first statement of Theorem~\ref{thm1} is a straightforward consequence of Theorem~\ref{prop0.2}.
Given that, the rigidity claim in conclusion is deduced in~\cite{B} from 
{\em Weinstein's inequality}
\begin{equation}
\label{weineq}
	\lambda_1^B(\Omega)\ge\lambda_2^D(\Omega)\,,
\end{equation}
which was conjectured in~\cite{We} and proved in~\cite{P,F},
and holds as an equality if and only if $\Omega$ is a disc.
That  rephrases an isoperimetric property: 
if the minimiser for $\lambda_1^B(\Omega)$ under area constraint
(which 
exists~\cite{AB} in the class of simply connected open sets) is smooth, then
it must be a disc.
We refer the reader to~\cite[Chap.~11]{H} and~\cite{KP}
for the detailed proof, due to N.B. Willms and H.F. Weinberger. In sketch, 
if $\Omega$ is a smooth open set 
then \eqref{corrente-2} is a necessary condition
of optimality in terms of the first buckling eigenfunction $\psi$,
and if $\Omega$ is not a disc
then,
by acting with infinitesimal rotations 
on the solution $w=\Delta\psi$ of the overdetermined problem \eqref{5},
one may produce
Dirichlet eigenfunctions relative to $\lambda$ with three nodal domains at least; in view of Faber-Krahn inequality, of the equality case in the
universal Ashbaugh-Benguria inequality, and of that in Weinstein's inequality, that would yield a contradiction unless $\lambda>\lambda_1^B(\Omega)$.

\begin{rmk}\label{inci}
Deducing \eqref{corrente-2} as a necessary condition of minimality of $\lambda_1^B(\Omega)$ among smooth open sets of given area requires simplicity, which at a first glance seems to imply a restriction
because $\lambda_1^B(\Omega)$ is not simple, in general~\cite{BP}. On the other hand, 
for a smooth planar simply connected domain, the optimality under area constraint is known
to imply the simplicity of the first eigenvalue, see e.g.\ the calculation at the beginning of
Sect.~4 in~\cite{HMP}. (We refer the
reader to~\cite{Le} for the analogous result for the clamped plate problem.) Conditions similar to
\eqref{corrente-2} hold also for stationary shapes, possibly even for higher eigenvalues, see~\cite{BL,BL2}.
\end{rmk}
\begin{rmk}\label{prermk}
A bounded
Lipschitz planar open set $\Omega$ is said to have the
{\em Pompeiu property}
if the function that vanishes identically is the only one that
integrates to zero on any isometric image of its~\cite{Z}.
The existence of eigenvalues 
$\lambda$ for \eqref{5} is known to be equivalent
to the failure of Pompeiu property~\cite{B}. Thus, in the conclusion
of Theorem~\ref{thm1}, in order to rule out the first option in the alternative one would need to answer in the affirmative the long-standing conjecture that Pompeiu's property fails if and only if $\Omega$ is the disc. 
\end{rmk}

Given $\lambda>0$
let $u$ be a non-trivial solution of \eqref{2}.
We recall that 
the 2D incompressible Navier-Stokes system
of \eqref{4}, with $u$ as initial datum, is well posed under assumptions that cover
the conservative case $f=0$, see~\cite{L0}.
Then, let 
 $v(x,t)$ be the corresponding global-in-time solution.
The associated vorticity $\omega=\partial_{x_1}v^2-\partial_{x_2}v^1$
solves
\begin{equation}
\label{vort-ev}
\begin{cases}
\partial_t\omega + (v\cdot\nabla)\omega = \nu \Delta\omega\,, & \qquad\text{in $\Omega\times(0,+\infty)$,}\\
\omega \perp \mathcal{H}\,, &\qquad \text{in $L^2(Q)$,}
\end{cases}
\end{equation}
where
\begin{equation}
\label{Hcal}
\text{ $Q=\Omega\times(0,\infty)$, and 
$\mathcal{H}=\big\{h\in L^2(0,\infty\mathbin{;}H^1(\Omega)) \mathbin{\colon}\Delta h(\cdot,t)=0\,, \text{for all $t>0$}\big\}$.} 
\end{equation}
The orthogonality condition in \eqref{vort-ev}
can be deduced from \eqref{4} by arguing as done in the proof of Proposition~\ref{prop0.1} for the linear stationary approximation, see also~\cite{QV}. By Proposition~\ref{prop0.1}, Eq.\ \eqref{vortex-1}
holds with  $w=\nabla\times u$.
If the additional condition \eqref{vortex-2} is also valid, then $w$ solves the overdetermined problem \eqref{5}, and we have
\[
\begin{cases}
(v\cdot\nabla)\omega = 0\,, & \qquad\text{in $\Omega\times(0,+\infty)$,}\\
e^{\nu\lambda t}\omega(\cdot,t) = {\rm const}&\qquad\text{for all $t>0$},\\
\frac{\partial\omega}{\partial n}(\cdot,t)=0\,,&
\qquad\text{for all $t>0$,}%\\
\end{cases}
\]
%and
\[
	\omega(x,t) =e^{-\nu \lambda t} w(x)\,,\qquad \text{for }(x,t)\in \Omega\times(0,+\infty)\,.
\]
Then, by observing that the function \eqref{h} introduced in Proposition~\ref{prop0.1} is the harmonic extension of the boundary vorticity values, we
infer the following result.

\begin{theorem}\label{prop1}Let $\Omega$ be
simply connected, let $\lambda>0$, let 
$u$
be a non-trivial solution of \eqref{2}, and let $\psi$ and $w$ be corresponding
solutions of \eqref{1} and \eqref{3}.
A necessary and sufficient condition that either
condition \eqref{corrente-2} of \eqref{vortex-2} (and hence both) be valid
is that 
\begin{equation}
\label{neumann}
	\frac{\partial p}{\partial n}=0\,,\qquad\text{on $\partial\Omega$}
\end{equation}
or, equivalently, that
$\nabla p=0$ in $\Omega$. 
If this happens and  $f=0$, then for any $\nu\in(0,+\infty)$ 
\begin{equation}
\label{cell}
v(x,t)= e^{-\nu\lambda t}u(x)\,,
\end{equation}
is the solution of \eqref{4} with $v(\cdot,0)=u$;
if $\Omega$ is of class $C^{2,\alpha}$ then the converse also holds.
\end{theorem}

Going back to our initial question,
pairing Theorem~\ref{thm1} and Theorem~\ref{prop1}
gives a positive answer at least in the case
of a fluid subject to conservative forces 
in a smooth simply connected domain with initial velocity given by a first
eigenfunction for \eqref{2}.
Also, as said, for the incompressible 2D Navier-Stokes system \eqref{4} in a smooth bounded domain 
the validity of homogeneous Neumann boundary conditions
 for the pressure is equivalent to a constant vorticity
 on the walls: by Theorem~\ref{prop1}, 
in the case of the linear stationary eigenvalue problem \eqref{2}
 the equivalence holds irrespective of the regularity of the boundary.

Incidentally, another consequence is that if
$\Omega\subset\mathbb R^2$
is a $C^{2,\alpha}$ simply connected bounded
open set, then the dissipation estimate for a solution $v$
of \eqref{4}, namely
\begin{equation}
\label{dissipation}
\sup_{t>0}  e^{\nu \lambda_1^S(\Omega) t}
\int_\Omega |v(x,t)|^2\,dx \le
\int_\Omega |v(x,0)|^2\,dx\,,
\end{equation}
holds as an equality only if $\Omega$ is a disc,
provided that the least eigenvalue of Stokes operator 
\begin{equation}
\label{lambda1S}
\lambda_1^S(\Omega) = \min_{u\in H^1_0(\Omega\mathbin{;}\mathbb R^2)}
\left\{\frac{\displaystyle\int_\Omega |\nabla u|^2\,dx}{\displaystyle
\int_\Omega|u|^2\,dx}\mathbin{\colon} \text{$\nabla\cdot u=0$ in $\Omega$}
\right\}
\end{equation}
is simple. 

By Theorem~\ref{prop1},
a statement equivalent to the conjecture
recalled in Remark~\ref{prermk} is that, with the exception of discs, smooth and bounded simply connected planar domains do not support cellular flows of the form
\eqref{cell}, with $\lambda >0$ and $u$ as in 
\eqref{2}, solving the 2D incompressible Navier-Stokes equations \eqref{4} with external forces coming from a potential.

The results in this note are much more elementary and only concern the special case in which $\lambda$ is takes
the special value \eqref{lambda1S}.
That was studied in~\cite{HMP}
from the perspective of spectral optimisation.
By definition (see Proposition~\ref{propSC} below), we have
$\lambda_1^B(\Omega)\ge \lambda_1^S(\Omega)$,
with equality if $\Omega$ is simply connected.
We remark two immediate implications of this fact.

\begin{rmk}[Isoperimetric property]
By the equality case in %Weinstein's inequality 
\eqref{weineq},
the isoperimetric property of the disc for \eqref{lambda1S}, were it valid, would confirm P\'olya and Szeg\H{o}'s belief that
{\em the disc minimises the least buckling eigenvalue among open sets of given area}. 
The local optimality proved in~\cite[Theorem~2]{HMP} is consistent
with this long standing conjecture. 
\end{rmk}

\begin{rmk}[Optimality conditions]\label{necopt}
Owing to
Theorem~\ref{thm1}, Theorem~\ref{prop1},
and Remark~\ref{inci}, if $\Omega$ is a bounded, smooth, and simply connected planar open set and
$\lambda_1^S(\Omega)$ is minimal among open sets of given area,
then the corresponding eigenfunction $u$ (unique up to
proportionality) is a pressureless solution of \eqref{2},
compare with
\cite[Theorem~4]{HMP}.\end{rmk}
One also recovers the following known fact:
if $\Omega\subset\mathbb R^2$ is a minimiser for $\lambda_1^S(\Omega)$ under
area constraint
{\em that is smooth, bounded, and simply connected}, 
then it must be a disc. A minimiser 
 under area constraint for $\lambda_1^S(\Omega)$
is known to exist in the class of quasi-open sets~\cite[Theorem~1]{HMP}. Unfortunately, the regularity issue lacks a solution and so
the isoperimetric property of the disc remains
an open problem for \eqref{lambda1S}, as well as for \eqref{lambda1B}.
\medskip

Eventually, we point out the difference with the 3D case, in which the ball fails the necessary minimality condition on the second shape variation sign~\cite[Theorem~3]{HMP}. In fact, in~\cite{HMP} it is also proved that the boundary
of any smooth minimiser  would belong to the homeomorphism class of the torus, and on that grounds the solution is conjectured to be the singular axis-symmetric 
domain obtained from a ball by bringing the poles to contact.

As for the rigidity of Neumann pressure boundary conditions
in three dimensions, around a {\em smooth} domain
the normal boundary pressure gradient for solutions of \eqref{4} 
reads as
\[
\frac{\partial p}{\partial n} = 
\mathbb{I}(v-n(v\cdot n)\mathbin,v-n(v\cdot n))
-\nu\, {\nabla}_{\tau}\times [(\nabla \times v)\times n]\,,
\]
where the (scalar) curl is in the local tangential coordinates 
$\tau=({\tau_1,\tau_2})$ that complete the outward unit normal $n$ to a positively oriented
orthonormal system in $\mathbb R^3$.
In the latter, one still gets rid of the first fundamental form $\mathbb I(\cdot\mathbin,\cdot)$  by the
no-slip condition, like in two dimensions, but equating
to zero the term proportional to viscosity has less obvious implications about the tangential vorticity, also depending on the topology of the
boundary.

Eventually, irrespective of the dimension,
we do not expect the results presented here to transfer to the vanishing viscosity limit of \eqref{4}. Indeed, formally, for
Euler equations we have $\nu=0$ and so no information on the boundary
vorticity seems to be obtainable from a homogeneous Neumann condition for the pressure at the boundary; also, the curvature term appearing in the formula for the normal derivative of the pressure does not clear off because of the tangential shear.

\section{Secondary variables}
We recall that $\Omega$, unless otherwise specified, will
be assumed to be a planar Lipschitz open set of finite area.
We begin with 
the proof of Proposition~\ref{prop0.1}, which concerns in particular the
equivalence of the three problems in the triptych
\begin{equation*}
\begin{minipage}{.25\textwidth}
\centering
\[
\begin{cases}
\nu\Delta^2\psi+\Delta \varphi=0\,, & \\
\psi=0\,, \text{ on $\partial\Omega$} & \,\\
\frac{\partial\psi}{\partial n}=0\,,\text{ on $\partial\Omega$}& 
\end{cases}
\]
\end{minipage}
\begin{minipage}{.25\textwidth}
\centering
\[
\begin{cases}
\nu\Delta u+ f=\nabla p\,, & \\
\nabla\cdot u=0\,,  & \\
u=0\,,\text{ on $\partial\Omega$}& 
\end{cases}
\]
\end{minipage}
\begin{minipage}{.3\textwidth}
\centering
\[
\begin{cases}
\nu\Delta w+\nabla\times f=0\,, & \\
\int_\Omega w\,h\,dx=0\,,  & \\
\quad \text{ for all harmonic $h\in H^1(\Omega)$,}& 
\end{cases}
\]
\end{minipage}
\end{equation*}
where we recall that \eqref{subASS} holds, i.e., $f \in L^2(\Omega\mathbin{;}\mathbb R^2)$ is such that $\nabla \times f\in H^{-1}(\Omega)$,
$\nabla\cdot f=0$ in $\Omega$, and
 $f\cdot n=0$ on $\partial\Omega$, and $\varphi\in H^1_0(\Omega)$ is such that $\Delta \varphi=\nabla\times f$.
 
 Here and henceforth, we always understand solutions of such elliptic problems the usual weak sense. More precisely, 
 we set $L^2_\sigma(\Omega;\mathbb R^2) = \{u\in L^2(\Omega;\mathbb R^2)\colon\nabla\cdot u=0\}$,
 we let
 $V_\sigma(\Omega)=H^1_0(\Omega;\mathbb R^2)\cap L^2_\sigma(\Omega) $,
 we define $\mathcal{J}$ as the closure in $H^1(\Omega)$,
 with respect to the $L^2(\Omega)$ topology,
 of the functions that are orthogonal
 in $L^2(\Omega)$ to all $H^1(\Omega)$ harmonic functions, and we write 
 that $\psi\in H^2_0(\Omega)$,
$u\in V_\sigma(\Omega)$, and 
 $w\in \mathcal{J}$ 
are weak solutions if
\begin{align}
\label{psiweak}
& \nu\int_\Omega \Delta \psi\Delta\phi\,dx  =\int_\Omega \nabla \varphi\cdot\nabla\phi\,dx\,, \qquad\text{for all $\phi\in H^2_0(\Omega)$,}\\
\label{uweak}
& \nu\int_\Omega \nabla u\,\colon\!\nabla v\,dx
=\int_\Omega f\cdot v\,dx\,,\qquad\text{for 
all $v\in V_\sigma(\Omega)$,}\\
\label{wweak}
& \nu	\int_\Omega \nabla w\cdot\nabla\xi\,dx
	=\big\langle \nabla\times f\mathbin,\xi\big\rangle\,,\qquad
	\text{for all $\xi\in\mathcal{J}$,}
\end{align}
respectively.  Details about weak solutions $u$ of the second problem can
be found, e.g., in~\cite{G}. 
 
We will make repeatedly use of the following Lemma,
in particular with the aim of relating the boundary conditions in the left-hand side to the non-local conditions in the right hand-side. We refer to~\cite{QV}
for a more general statement in the case of smooth open sets.

\begin{lemma}\label{lemma-orto}
Let $\Omega$ be Lipschitz and $\psi\in H^2(\Omega)\cap H^1_0(\Omega)$. Then $\psi\in H^2_0(\Omega)$ if and only if
\begin{equation}
\label{orthoha}
\int_\Omega h \Delta \psi\,dx =0\,,\qquad\text{for all
$h\in H^1(\Omega)$
with $\Delta h=0$ in $\Omega$.}
\end{equation}
\end{lemma}
\begin{proof}
Integrating by parts twice yields \eqref{orthoha} for $\psi\in C^\infty_0(\Omega)$, and hence for all $\psi\in H^2_0(\Omega)$ by
a density argument. Conversely, let $\psi \in H^2(\Omega)\cap H_0^1(\Omega)$
be orthogonal to all harmonic functions. Then, 
the vector field $\nabla\psi$ has divergence in $L^2(\Omega)$. Hence, its normal trace at the boundary
 $\frac{\partial\psi}{\partial n}\in H^{-1/2}(\partial\Omega)$ is
 in fact a well defined element of $L^2(\partial\Omega)$.
If $h\in H^1(\Omega)$ is harmonic,
\[
\int_{\partial\Omega}h\frac{\partial\psi}{\partial n}\,d\mathcal{H}^1 = \int_\Omega \nabla h \cdot \nabla \psi\,dx = 0\,,
\]
where we also used, in the first integration by parts, that $\Delta\psi$
is orthogonal to $h$ and, in the second one, that the boundary
trace of $\psi $ is zero. By the solvability of the Dirichlet problem for the Laplace equation for all boundary values in the trace space $H^{1/2}(\partial\Omega)$,
we deduce that the boundary trace $\frac{\partial\psi}{\partial n}$ 
is orthogonal to all elements of a dense subset
of $L^2(\partial\Omega)$ and hence must be zero.
\end{proof}

\subsection*{Proof of Proposition~\ref{prop0.1}}
We split the proof into steps.
\subsubsection*{From vortices to streams}
Let $w\in H^1(\Omega)$ be a weak solution of \eqref{vortex-1}. We let $\psi\in H^1_0(\Omega)$ be a weak solution of Poisson equation $\Delta\psi=w $.
Then in fact $\psi \in H^2(\Omega)$, by elliptic regularity~\cite[Theorem~8.12]{GT}.
By Lemma~\ref{lemma-orto},
the orthogonality of $w$ against harmonic functions implies
that $\psi\in H_0^2(\Omega)$, too. Then,
for every $\xi\in H^2_0(\Omega)$, with
$\xi\perp h $ in $L^2(\Omega)$ for all 
harmonic $h\in H^1(\Omega)$,  we have
\[
\begin{split}
\nu\int_\Omega \Delta\psi\,\Delta\xi\,dx
=\nu\int_\Omega w \Delta\xi\,dx =
-\nu\int_\Omega \nabla w\cdot\nabla\xi\,dx
=-\big\langle\nabla\times f\mathbin,\xi\big\rangle
=\int_\Omega\nabla\varphi\cdot\nabla\xi\,dx\,,
\end{split}
\]
where we also used that
 $\varphi\in H^1_0(\Omega)$ and $\Delta \varphi=\nabla\times f$.
Also, for all harmonic $h\in H^1(\Omega)$ 
\[
	\nu\int_\Omega \Delta\psi\,\Delta h\,dx=
	\int_\Omega\nabla\varphi\cdot\nabla h\,dx=0\,.
\]
Since we can write any element of $H^2_0(\Omega)$
in the form $\phi=\xi+h$, the equations sum up to
\eqref{psiweak}. Then, $\psi$ is
a weak solution of \eqref{corrente-1}.

\subsubsection*{From streams
to vortices}
We let $\psi\in H^2_0(\Omega)$ be a weak solution of \eqref{corrente-1}. By
\eqref{psiweak},
setting $w = \Delta \psi$ defines an element of $H^1(\Omega)$. By Lemma~\ref{lemma-orto}, $w$ is orthogonal to all $H^1(\Omega)$ harmonic functions. 
By \eqref{varphiASS}
we also have
\[
\nu\int_\Omega\nabla w\cdot\nabla\phi\,dx-\big\langle\nabla\times f\mathbin{,}\phi\big\rangle
=\nu \int_\Omega\nabla w\cdot\nabla\phi\,dx+
\int_\Omega\nabla\varphi\cdot\nabla\phi\,dx
\] 
where $\langle\cdot\mathbin{,}\cdot\rangle $ denotes the pairing between $H^1_0(\Omega)$ and its dual,
and $w = \Delta \psi$ implies
\[
\nu \int_\Omega\nabla w\cdot\nabla\phi\,dx+
\int_\Omega\nabla\varphi{\cdot}\nabla\phi\,dx
= -\nu\int_\Omega\Delta\psi\Delta \phi\,dx + \int_\Omega\nabla\varphi\cdot\nabla\phi\,dx\,,
\quad\text{for $\phi\in C^\infty_0(\Omega)$,}
\]
with no boundary term arising from partial integration. 
For every $\xi \in H^1(\Omega)$ orthogonal in $L^2(\Omega)$ to harmonic functions, by a density argument we then obtain \eqref{wweak}.
 Thus,
$w$ is a weak solution of \eqref{vortex-1}.

\subsubsection*{The last part of the statement}
We let $\Omega'\subset\Omega$ be an open
set and we let $\xi\in C^\infty_0(\Omega)$ with support contained in $\Omega'$. Then, we take
a weak solution 
$\psi\in H^2_0(\Omega)$ of \eqref{corrente-1}
and we let $\phi\in C^\infty_0(\Omega)$ be the solution
of  $\Delta \phi =\xi$ in  $H^1_0(\Omega')$.
Thus, we can integrate by
parts in \eqref{psiweak} and by doing so we get
\[
	\nu \int_\Omega \Delta\psi\,\xi\,dx
	-\int_\Omega \varphi\,\xi\,dx=0\,.
\]
As $\xi$ was arbitrary, it follows that \eqref{h} defines
a harmonic function. By differentiating in \eqref{h} we arrive at
\begin{equation}
\label{perph-weyl}
	\nabla^\perp h = \nu \Delta (\nabla^\perp \psi)
	+\nabla^\perp \varphi
	\,.
\end{equation}

Now, we let $u$ be defined as in
\eqref{perp}. Clearly, $u\in H^1_0(\Omega;\mathbb R^2)$.  Then, the laplacian in \eqref{perph-weyl}
is acting on $-u$. As for the last term,
it equals $-f$ up to the gradient of a
scalar function: indeed, we have
\[
\nabla\times (\nabla^\perp\varphi+f) =- \Delta\varphi +\nabla\times f=0
\]
in the distributional sense in $\Omega$, and $\Omega$ is simply connected. 
Summing up, the right hand side of \eqref{perph-weyl} is $-\nu\Delta u-f + \nabla \tilde q$, for an appropriate $\tilde q\in H^1(\Omega)$.
Since $\Omega$ is simply connected,
the harmonic function  $h$ admits a harmonic conjugate $q$, i.e., for the left hand side of \eqref{perph-weyl}
we have $\nabla^\perp h = \nabla q$ for a suitable $H^1(\Omega)$ harmonic function $q$. By setting
$p=q-\tilde  q$ we have
\begin{equation}
\label{ptwise}
	\nabla p =\nu \Delta u+ f\,,
\end{equation}
in the sense of distributions in $\Omega$. Since
 $u \in H^1_0(\Omega;\mathbb R^2) $, with $\nabla\cdot u=0$
 (as $\psi\in H^2_0(\Omega)$ and \eqref{perp} holds),
 and $p\in H^1(\Omega)$, by construction, it follows that $u$ is a weak solution of \eqref{Re-zero}.

Conversely, if 
 $u \in H^1_0(\Omega;\mathbb R^2) $, with $\nabla\cdot u=0$, is a weak solution of \eqref{Re-zero} then by
classical interior estimates~\cite[Chap.~3]{L} for  \eqref{uweak}, Eq.\ \eqref{ptwise} holds in the
classical sense. By the uniqueness for the Dirichlet problem
\eqref{varphiASS}, we know that $f=-\nabla^\perp \varphi$. Then, with \eqref{perp} in force, from \eqref{ptwise} we arrive at \eqref{CR}, as desired.
\qed

\section{Overdetermined eigenvalue problems}

\subsection*{Proof of Theorem~\ref{prop0.2}}
The overdermined problems
\eqref{corrente} and \eqref{vortex}
are equivalent because, 
by Proposition~\ref{prop0.1},
so are 
 the problems
\eqref{corrente-1} and \eqref{vortex-1}.

Then, we assume that $\Omega$ is simply connected, 
$\lambda>0$ and $u\in H^1_0(\Omega\mathbin;\mathbb R^2)$ solves  \eqref{Re-zero}. Afterwards, we let $\psi\in H^2_0(\Omega)$ and $w\in H^1(\Omega)$,
orthogonal in $L^2(\Omega)$ to all $H^1(\Omega)$ harmonic functions,
be corresponding solutions of of \eqref{corrente-1} and \eqref{vortex-1}.

By \eqref{varphiASS}, the function \eqref{h} is the harmonic extension of the boundary trace of $w$. By the maximum principle, the latter is constant if and only if so is $h$ in the interior. Therefore, in view of Eq.\ \eqref{CR} in Proposition~\ref{prop0.1},
conditions \eqref{corrente-2} and \eqref{vortex-2}  are equivalent to \eqref{Re-zero-plus}.\qed

\subsection*{Proof of Theorem~\ref{thm1}}
Let $\lambda>0$. By applying Theorem~\ref{prop0.2} to the case $f=\nu^{-1}\lambda u$, the eigenvalue problems
\eqref{1}, \eqref{2}, and \eqref{3} are equivalent and so are the corresponding overdetermined problems
\begin{equation}\label{eigentryptich}
\begin{minipage}{.25\textwidth}
\centering
\[
\begin{cases}
\Delta^2\psi+\lambda\Delta\psi=0\,, & \\
\psi=0\,, \text{ on $\partial\Omega$,} & \,\\
\frac{\partial\psi}{\partial n}=0\,,\text{ on $\partial\Omega$,}& \\
\Delta\psi = {\rm const} \,,\text{ on $\partial\Omega$,} &
\end{cases}
\]
\end{minipage}
\begin{minipage}{.25\textwidth}
\centering
\[
\begin{cases}
\Delta u+ \lambda u=\nabla p\,, & \\
\nabla\cdot u=0\,,  & \\
u=0\,,\text{ on $\partial\Omega$,}&\\
\frac{\partial p}{\partial n} =0 \,, \text{ on $\partial\Omega$,} 
\end{cases}
\]
\end{minipage}
\begin{minipage}{.3\textwidth}
\centering
\[
\begin{cases}
\Delta w+\lambda w=0\,, & \\
\int_\Omega w\,h\,dx=0\,,  & \\
\quad \text{ for all harmonic $h\in H^1(\Omega)$,}&\\
w = {\rm const}\,,\text{ on $\partial\Omega$.} & 
\end{cases}
\]
\end{minipage}
\end{equation}

Now we prove that the problem to the right is equivalent to \eqref{5}.
To do so,
we let
\begin{equation}
\label{weakshi}
w \in H^1(\Omega)\qquad\text{and}\quad	
\int_\Omega \nabla w\cdot\nabla \phi\,dx=
\lambda \int_\Omega  w\,\phi\,dx\,,
	\quad\text{for all $\phi\in H^1(\Omega)$,}
\end{equation}
and we assume that
$w$ has a constant boundary trace. If $h\in H^1(\Omega)$ and $\Delta h=0$ then
\begin{equation}
\label{weakshi2}
\int_\Omega \nabla h\cdot\nabla w\,dx = \int_{\partial\Omega}  w\frac{\partial h}{\partial n}\,d\mathcal{H}^1= {\rm const}\,{\cdot}\!\!\int_\Omega\Delta h\,dx=0\,,
\end{equation}
where in second and in the third equality we used that $h$ is harmonic. By pairing \eqref{weakshi} and \eqref{weakshi2} 
we see that $w$ solves the problem to the right in \eqref{eigentryptich}. 

For the converse, we let $w\in H^1(\Omega)$
be a solution of the
problem to the right in \eqref{eigentryptich}
and
we fix $\phi \in H^1(\Omega)$,
which we can write in the form $\phi=\xi+h$
with $\xi\in H^1(\Omega)$ orthogonal in $L^2(\Omega)$ to all 
harmonic $H^1(\Omega)$ functions, and $h\in H^1(\Omega)$ harmonic.
By assumption
we have
\[
\int_\Omega \nabla w\cdot \nabla\xi\,dx =\lambda \int_\Omega w\,\xi\,dx\,.
\]
In fact, we also have
\[
\int_\Omega \nabla w\cdot \nabla h\,dx =\lambda \int_\Omega w\,h\,dx\,,
\]
because both the integrals in the latter equal zero:
the first one
by \eqref{weakshi2} and the second one by construction.
Since $\phi\in H^1(\Omega)$ was arbitrary, by summing up 
it follows that \eqref{weakshi} holds, as desired.
For the last statement, we refer to the proof of~\cite[Theorem~11.3.7]{H}.\qed

\begin{rmk}
We have the equivalence between the three overdetermined vorticity problems
\[
\begin{minipage}{.25\textwidth}
\centering
\[
\begin{cases}
\Delta w+\lambda w=0\,, & \\
w = {\rm const}\,,\text{ on $\partial\Omega$,} & \\
\frac{\partial w}{\partial n}=0\,, 
\quad \text{on $\partial\Omega$,}&
\end{cases}
\]
\end{minipage}
\begin{minipage}{.25\textwidth}
\centering
\[
\begin{cases}
\Delta w+\lambda w=0\,, & \\
w\in \mathcal{J},  & \\
\frac{\partial w}{\partial n}=0\,,\text{ on $\partial\Omega$,} & 
\end{cases}
\]

\end{minipage}
\begin{minipage}{.3\textwidth}
\centering
\[
\begin{cases}
\Delta w+\lambda w=0\,, & \\
w\in \mathcal{J},  & \\
w = {\rm const}\,,\text{ on $\partial\Omega$.} & 
\end{cases}
\]
\end{minipage}
\]
Here, $\mathcal{J}$ is comprised of all 
$H^1(\Omega)$ functions
that are orthogonal in $L^2(\Omega)$ to  harmonic functions.
While proving Theorem~\ref{thm1}, we have
seen that the left and the right hand side are equivalent. The equivalence of the latter and of the problem midway follows by a similar argument and by applying a boundary DuBois-Reymond type lemma in local coordinates.
\end{rmk}

We turn our attention to the two-dimensional
incompressible Navier-Stokes system \eqref{4}.
\begin{rmk}\label{wprNS}
The well-posedness of the Cauchy problem for \eqref{4}
is due to~\cite{L0} 
under
assumptions covering the case of conservative forces and of divergence-free $H^1(\Omega;\mathbb R^2)$ initial data.
Moreover, when $f=0$ and the solution $v$
of \eqref{4} takes the
form \eqref{cell} we have in particular that $v(x,0)=u$.
Since by elliptic regularity the latter belongs to $H^2(\Omega)$, the function $v$ is in fact a classical solution of \eqref{4},
see~\cite[Theorem~7 in Chap.~6]{L}.
\end{rmk}

In order to prove Theorem~\ref{prop1},
we need the following lemmas, in which
we are implicitly assuming that $f=0$ and that \eqref{cell}
holds.

\begin{lemma}\label{fromu}
If $u$ solves \eqref{2} and $\nabla p=0$, then $v$ solves \eqref{4}. 
\end{lemma}
\begin{lemma}\label{fromv}
If $\Omega\in C^{2,\alpha}$ and $v$ solves \eqref{4}, then $u$ solves \eqref{2} with $\frac{\partial p}{\partial n}=0$.
\end{lemma}

\begin{proof}[Proof of Lemma~\ref{fromu}]
Let $u$ solve \eqref{2}, with $\nabla p=0$, and let
 $w=\nabla\times u$. Then, \eqref{cell}
is the solution
of the incompressible heat PDE system
\begin{equation}
\label{heatsys}
\begin{cases}
	\partial_tv = \nu \Delta v & \qquad \text{in } Q\\
	\nabla\cdot v=0 & \qquad \text{in } Q\\
	v=0 &\qquad \text{in } \partial_LQ \,,
\end{cases}
\end{equation}
where $Q=\Omega\times(0,+\infty)$ 
and $\partial_LQ = \partial\Omega\times(0,+\infty)$, with
$v(\cdot,0)=u$.

By arguing as done to prove Proposition~\ref{prop0.1}, we see that the vorticity $\omega(x,t)= e^{-\nu\lambda t} w(x)$
belongs to $L^2(0,\infty;\mathcal{J})\cap H^1(0,\infty;L^2(\Omega))$, with $\mathcal{J}$ being the $L^2(\Omega)$-orthogonal
complement in $H^1(\Omega)$ to the
space 
of $H^1(\Omega)$
harmonic functions, and it solves the initial value problem \eqref{vort-ev} with $\omega(\cdot,t)=w$.

By interior elliptic estimates, $u$ is smooth, 
and we have
$\nabla^\perp w=-\Delta u$. Hence and from \eqref{2} we get $(u\cdot\nabla)w=(\nabla p-\lambda u)\cdot\nabla \psi$,
where $u=-\nabla^\perp\psi$ and $\psi$ solves \eqref{1}. Since $\nabla p =0$ and $u\cdot \nabla \psi=0$, 
we have $(u\cdot\nabla)w=0$ in $\Omega$ and that implies  \begin{equation}
\label{nostretch}
(v\cdot \nabla) \omega=0\,,\qquad \text{in } Q\,.
\end{equation}
Pairing \eqref{vort-ev} and \eqref{nostretch}, and recalling
$\omega=\nabla\times u$,
we deduce the existence of a pressure 
$p$ (uniquely defined up to additive constants) 
for which \eqref{4} holds with $f=0$.
\end{proof}

\begin{proof}[Proof of Lemma~\ref{fromv}]
We let $u$ be a weak solution of \eqref{2} with pressure $ p$
and we assume \eqref{cell}
to define a solution of Naver-Stokes system
\eqref{4}, with $f=0$.
Then the equation $\partial_t\omega+(v\cdot\nabla)\omega=\nu\Delta\omega$ holds for
$\omega=e^{-\nu\lambda t}w$, where $w=\nabla\times u$. Since we also have $\partial_t\omega=\nu\Delta\omega$,
it follows that $-e^{-2\nu\lambda t}\nabla\psi\cdot\nabla p=e^{-2\nu\lambda t}(u\cdot\nabla)w=(v\cdot \nabla)\omega=0$ for all $t>0$, where $\psi\in H^2_0(\Omega)$
is such that $u=-\nabla^\perp u$.  If $\Omega$ is of class $C^{2,\alpha}$ then by~\cite[Theorem~5, Chap.~3]{L}
the equations hold at the boundary, which is a level set of $\psi$. Thus, $\frac{\partial p}{\partial n}=0$.
\end{proof}
\subsection*{Proof of Theorem~\ref{prop1}} Let $p$ be the pressure associated with $u$. We recall that $p$ is unique
up to an additive constant and that it is a harmonic function. Hence, the first statement follows by Theorem~\ref{prop0.2}. Then, we fix $\nu>0 $, we define $v$ as in \eqref{cell}. The
first implication in the last statement follows by Lemma~\ref{fromu}. If in addition $\Omega$ is smooth, the
converse implication also holds thanks to Lemma~\ref{fromv}. \qed

\subsection*{The ground states}
We conclude by an elementary remark about the relation
between the the first eigenvalues
in the buckling problem and in Stokes system.
\begin{proposition}\label{propSC}
\(
\lambda_1^B(\Omega) 
\ge\lambda_1^{S}(\Omega)
\),
with equality if $\Omega\subset\mathbb R^2$ is simply connected.
\end{proposition}

For the proof we need the following Lemma.
\begin{lemma}\label{byparts}
Let $\psi\in H^2_0(\Omega)$ and
 $u=\nabla^\perp \psi $. Then $u\in H^1_0(\Omega;\mathbb R^2)$,
with $\nabla\cdot u=0 $, and 
\[
\frac{\displaystyle\int_\Omega|\nabla u|^2\,dx}{\displaystyle\int_\Omega|u|^2\,dx}
=\frac{\displaystyle\int_\Omega(\Delta\psi)^2\,dx}{\displaystyle\int_\Omega|\nabla\psi|^2\,dx}
\]
\end{lemma}
\begin{proof}
One uses that $\partial_{x_1x_2}^2\psi-\partial_{x_1x_1}^2\psi\partial_{x_2x_2}^2\psi$ 
integrates to zero in $\Omega$.
\end{proof}

\begin{proof}[Proof of Proposition~\ref{propSC}]
We let $\varepsilon>0$. 
Then, we take $\psi\in C^\infty_0(\Omega)$ with
\[
 \lambda_1^B(\Omega)+\varepsilon\ge \frac{\displaystyle\int_\Omega(\Delta\psi)^2\,dx}{\displaystyle\int_\Omega|\nabla\psi|^2\,dx}\,.
 \]
Note that
 $u=\nabla^\perp \psi$
 is  admissible for the definition of $\lambda_1^S(\Omega)$. Then, $\lambda_1^B(\Omega)\ge\lambda_1^S(\Omega)$ by Lemma~\ref{byparts}, because $\varepsilon>0$ is arbitrary.

Now we assume that $\Omega$ is simply connected. Recall that $C^\infty_0(\Omega;\mathbb R^2) $ is dense in $H^1_0(\Omega;\mathbb R^3)$.  Then, by arguing as done above,
we find $u\in C^\infty_0(\Omega;\mathbb R^2)$, with $\nabla\cdot u=0$, such that
\[
 \lambda_1^S(\Omega)+\varepsilon\ge \frac{\displaystyle\int_\Omega|\nabla u|^2\,dx}{\displaystyle\int_\Omega|u|^2\,dx}\,.
\]
Since $\Omega$ is simply connected,
there exists $\psi\in H^2_0(\Omega)$ with $-\nabla^\perp \psi = u$
and we conclude again by Lemma~\ref{byparts},
because $\varepsilon>0$ is arbitary.
\end{proof}

\end{document}